\documentclass[10pt,notitlepage,reqno]{amsart}

\usepackage{verbatim}
\usepackage{amssymb,amsmath,amsthm,mathrsfs,mathtools}
\usepackage{amscd}
\usepackage[breaklinks=true]{hyperref}
\usepackage{url}
\usepackage{units}
\usepackage[usenames,dvipsnames]{xcolor}
\usepackage{graphicx}
\usepackage[all]{xy}
\usepackage{amsmath}
\usepackage{cleveref}
\usepackage{xy}
\hypersetup{
  colorlinks,
  linkcolor	=	Gray,
  urlcolor 	= Gray,
  citecolor	=	Gray,
  pdfauthor={Mirco Richter},
  pdfkeywords={Lie, infinity, Rinehart, Schouten, exterior, codifferential},
  pdftitle={Lie oo-algebras from Lie Rinehart pairs},
  pdfsubject={Lie, Rinehart, infinity},
  pdfpagemode=UseNone
}

\let\originalleft\left
\let\originalright\right
\renewcommand{\left}{\mathopen{}\mathclose\bgroup\originalleft}
\renewcommand{\right}{\aftergroup\egroup\originalright}

\newcommand{\bwedge}{\raisebox{0.3ex}{${\scriptstyle\bigwedge\,}$}}
\newcommand{\bbwedge}[1]{\raisebox{0.3ex}{${\scriptstyle \bigwedge}$}
 \ensuremath{^{\raisebox{-0.2ex}{${\scriptstyle #1}$}}\,}}

\newcommand{\smwedge}{{\scriptstyle \;\wedge\;}}
\newcommand{\ssmwedge}{{\scriptstyle \;\wedge}}
\newcommand{\ordinal}[1]{[$\;\!#1\,$]}
\newcommand{\R}{\mathbb{R}}
\newcommand{\Z}{\mathbb{Z}}
\newcommand{\N}{\mathbb{N}}

\newcommand{\NN}{\mathbb{N}_0}

\newcommand{\I}{\infty}

\newcommand{\BIGOP}[1]{\mathop{\mathchoice
{\raise-0.22em\hbox{\huge $#1$}}
{\raise-0.05em\hbox{\Large $#1$}}{\hbox{\large $#1$}}{#1}}}

\def\clap#1{\hbox to 0pt{\hss#1\hss}}

\theoremstyle{plain}
\newtheorem{theorem}{Theorem}[section]
\newtheorem{prop}[theorem]{Proposition}

\newtheorem{corollary}[theorem]{Corollary}

\newtheorem{definition}[theorem]{Definition}

\theoremstyle{remark}
\newtheorem*{remark}{Remark}
\newtheorem{example}{Example}

\title{Lie $\I$-algebras from Lie - Rinehart Pairs}
\author{Mirco Richter }
\date {\today}
\thanks{email: \href{mailto:mirco.richter@email.de}{\tt mirco.richter@email.de}}

\begin{document}

\begin{abstract}  
We generalize the Schouten calculus of multivector fields to 
commutative Lie Rinehart pairs and define a non negatively graded 
Lie $\I$-algebra on their exterior power.
\end{abstract}
\maketitle
\section{Introduction}Lie Rinehart pairs generalize the algebraic structure of
vector fields and smooth functions to commutative algebras and Lie algebras, 
which are some kind of modules with 
respect to each other. In particular any Lie algebra together with its
underlying field defines a Lie Rinehart pair.

Given such a pair, we look at its exterior algebra, that is the
exterior power of the Lie partner seen as a module with respect to the commutative 
algebra. In case of vector fields and smooth functions, this is precisely 
the algebra of \textit{multivector fields}.

The exterior power of an ordinary Lie algebra has many structures.
Scientist with a more algebraic background eventually look on it as a 
particular codifferential graded \textit{coalgebra}, where the
codifferential encodes the Lie algebra structure \cite{MM}, 
while scientist coming more from differential
geometry likely see it as another graded Lie algebra with respect 
to the \textit{Schouten-Nijenhuis} bracket \cite{PM}.

We will show that the codifferential approach is not necessarily well 
defined with respect to the additional module structure, but that the 
Schouten-Nijenhuis bracket is natural.

Then we provide
a non negatively graded Lie $\I$-algebra  
on the exterior power, which comes with a natural injection of
the original Lie Rinehart pair. In contrast to ordinary Lie theory this injection 
is not a single map, but a whole sequence of maps. 
Those functions are usually called
\textit{weak} Lie $\I$-morphisms and we explain them in more detail in appendix A.

On the level of objects, this will merely be a shift in perspective, 
but we get a considerably richer theory on the level of morphisms, 
since we gain access to functions, which are much more flexible then 
ordinary map of (graded) Lie algebras.
\section{The Schouten-Nijenhuis Algebra of a Lie Rinehart pair} We start our work
with a short introduction to Lie Rinehart pairs. 
We look at their exterior powers and show that in
contrast to ordinary Lie algebras, in general there is no well 
defined (co)differential in this setting anymore. Then we introduce the
Schouten-Nijenhuis bracket, well known from differential geometry. 
\subsection{Lie Rinehart pairs}
In what follows $\mathfrak{g}$ will always be a real Lie algebra, that is a $\R$-vector space together with an antisymmetric, bilinear map,
\begin{equation}
[\cdot,\cdot]: \mathfrak{g}\times\mathfrak{g}\to\mathfrak{g}
\end{equation}
called \textbf{Lie bracket}, such that for any three vector $x_1$, $x_2$ and $x_3\in\mathfrak{g}$ the \textbf{Jacobi identity}
$[x_1,[x_2,x_3]]+[x_2,[x_3,x_1]]+[x_3,[x_1,x_2]]=0$ is satisfied.

In addition $A$ will always be a real
associative and commutative algebra with unit, 
that is a $\R$-vector space together with an associative and commutative, bilinear map
\begin{equation}
\cdot : A \times A \to A
\end{equation}
called multiplication and a unit $1_A\in A$.
According to a better readable text, 
we frequently suppress the symbol of the multiplication in $A$ and just
write $ab$ instead of $a\cdot b$.

Moreover $Der(A)$ will be the Lie algebra of derivations of $A$, 
that is the vector space of linear endomorphisms of $A$, 
with $D(ab)=D(a)b+aD(b)$ and Lie bracket
$[D,D'](a):=D(D'(a))-D'(D(a))$ for any $a,b\in A$ and $D,D'\in Der(A)$.

Before we get to Lie Rinehart pairs, it is handy to define Lie algebra modules first:
\begin{definition}[Lie algebra module]
Let $\mathfrak{g}$ be a real Lie algebra, $A$ an $\R$-algebra and
$D:\mathfrak{g}\to Der(A)$ a Lie algebra morphism. Then $A$ is called
a \textbf{Lie algebra module} (or just $\mathfrak{g}$-module) and $D$
is called the $\mathfrak{g}$\textbf{-scalar multiplication}.
\end{definition}
Now a \textit{Lie Rinehart pair} is nothing but a Lie algebra 
and an associative algebra, 
each of them being a module with respect to the other, such that 
a particular compatibility equation of their multiplications is satisfied: 
\begin{definition}[Lie Rinehart Pair] Let $A$ be an associative and 
commutative algebra with unit, $\mathfrak{g}$ a Lie algebra and 
$\cdot_A: A\times \mathfrak{g} \to \mathfrak{g}$ as well as
$D: \mathfrak{g}\to Der(A)\;;\; x\mapsto D_x$ maps, such that 
$A$ is a $\mathfrak{g}$-module with $\mathfrak{g}$-scalar multiplication
$D$, the vector space $\mathfrak{g}$ is
an $A$-module with $A$-scalar multiplication $\cdot_A$ and 
the \textbf{Leibniz rule}
\begin{equation}
[x,a\cdot_A y] = D_x(a)\cdot_A y + a \cdot_A[x,y]
\end{equation}
is satisfied for any $x,y\in\mathfrak{g}$ and $a\in A$. Then
$\left(A,\mathfrak{g}\right)$ is called a \textbf{Lie Rinehart pair}.
\end{definition}
This can be defined more general over arbitrary ground rings with unit and also
with respect to non commutative algebras $A$. We stick to the commutative 
situation, since we need that property to define exterior powers later on.
A more general introduction can be found in \cite{JH} and in the references therein. 

The two most extreme examples coming from commutative algebras on one side and
Lie algebras on the other:
\begin{example}
For any commutative and associative algebra with unit $A$, a 
Lie  pair is given by $(A, Der(A))$, together with the standard 
action of $Der(A)$ on $A$ and the standard $A$-module structure of $Der(A)$.
\end{example}
\begin{example}
Any real Lie algebra $\mathfrak{g}$ is a $\R$-module
with respect to its ordinary scalar multiplication 
and togeter with the \textit{trivial action} of $\mathfrak{g}$ on $\R$, given by
$$
D:\mathfrak{g}\times\R \to\R\;;\; (x,\lambda)\mapsto D_x(\lambda):=0,
$$
the pair $\left(\R,\mathfrak{g}\right)$ becomes a Lie Rinehart pair.
\end{example}
As mentioned before the archetypical example is 
provided by smooth functions and vector fields on a differentiable manifold:
\begin{example}
Let $M$ be a differentiable manifold, $C^\I(M)$ the algebra of smooth, 
real valued functions and $\mathfrak{X}(M)$ the Lie algebra of vector
fields on $M$. $\mathfrak{X}(M)$ is a $C^\I(M)$-module and 
vector fields acts as derivations on smooth functions, that is the
map
$$D:\mathfrak{X}(M)\times C^\I(M) \to C^\I(M)\;;\;
 (X,f) \mapsto D_X(f):= X(f)
$$
satisfies the equation
$D_X(fg)=D_X(f)g + fD_X(g)$. Moreover the Leibniz rule
$[X,fY]=D_X(f)Y+f[X,Y]$ holds and it follows that
$(C^\I(M),\mathfrak{X}(M))$ is a Lie Rinehart pair.
\end{example}
The Lie structure can be extended into a graded Lie algebra on the
direct sum of the partners, concentrated in degrees zero and one. 
This appears in \cite{GR}:
\begin{definition}[Associated Lie algebra]
Let $(A,\mathfrak{g})$ be a Lie Rinehart pair. Its 
\textbf{associated} (graded) Lie algebra is the direct
sum $A\oplus\mathfrak{g}$ seen as a graded vector space,
with $A$ concentrated in degree zero, $\mathfrak{g}$ concentrated
in degree one and Lie bracket defined by
\begin{equation}
\begin{array}{crcl}
[\cdot,\cdot]: & A\oplus\mathfrak{g} \times A\oplus\mathfrak{g} &\to&
 A\oplus\mathfrak{g}\\
 & \left((a,x),(b,y)\right) &\mapsto& \left(D_x(a)+D_y(b),[x,y]\right)\;.
\end{array}
\end{equation}
\end{definition}
In particular this means, that we can see any Lie Rinehart pair as a 
graded Lie algebra and that it makes sense to talk about (graded)
Lie algebra morphisms in their context. The following proposition 
justifies the definition:
\begin{prop}$(A\oplus\mathfrak{g},[\cdot,\cdot])$ is a graded Lie
algebra.
\end{prop}
\begin{proof}
$A\oplus\mathfrak{g}$ is a graded vector space by definition.
To see graded symmetry of the bracket, we only need to 
consider mixed expressions, where we compute  
$[(a,0),(0,x)]=(D_x(a),0)=[(0,x),(a,0)]=
(-1)^{|a||x|}[(x,0),(a,0)]$ on scalars $a\in A$ and vectors
$x\in\mathfrak{g}$.   

To see the graded symmetric Jacobi identity, observe that it has to vanish,
whenever at least two arguments are scalars, since the left side of the identity
is an expression, homogeneous of degree $-2$. If all arguments are vectors, it becomes the usual Jacobi identity of $\mathfrak{g}$ and the
remaining cases are seen from
$D_y\circ D_x-D_x\circ D_y +D_{[x,y]}=0$.
\end{proof}
Morphisms of Lie Rinehart pairs are pairs
of appropriate algebra maps, which interact properly with
respect to the additional module structures \cite{GR}:
\begin{definition}[Lie Rinehart Morphism]
Let $(A,\mathfrak{g})$ and $(B,\mathfrak{h})$ be two Lie Rinehart pairs.
A \textbf{morphism of Lie Rinehart pairs} is a pair of maps $(f,g)$, such
that $f:A\to B$ is a morphism of associative and commutative, real algebras
with unit, $g:\mathfrak{g}\to\mathfrak{h}$ is a morphism of Lie
algebras and the equations
\begin{equation}\label{LR-morph}
\begin{array}{ccc}
g(a\cdot_A x)=f(a)\cdot_B g(x)
& and &
f(D_x(a))=D_{g(x)}(f(a))
\end{array}
\end{equation} 
are satisfied for any $a\in A$ and $x\in \mathfrak{g}$.  
\end{definition}
This is the correct definition of a
morphism in the setting of Lie Rinehart pairs, since all structure is respected
properly:
\begin{corollary}
Let $(f,g):(A,\mathfrak{g})\to(B,\mathfrak{h})$ be a morphism of
Lie Rinehart pairs. The image $(f(A),g(\mathfrak{g}))$ is a
Lie Rinehart pair and 
$(f,g):A\oplus\mathfrak{g}\to B\oplus\mathfrak{h}$ is a morphism of 
graed Lie algebras.
\end{corollary}
\begin{proof} The first structure equation of (\ref{LR-morph}) implies that 
the vector space $g(\mathfrak{g})$ is a $f(A)$-module and the second
that $f(A)$ is a $g(\mathfrak{g})$-module. To see the Leibniz equation, compute
\begin{align*}
&[g(x),f(a)\cdot_B g(y)] = [g(x),g(a\cdot_A y)] = g([x,a\cdot_A y])
=g(D_x(a)\cdot_A y + a \cdot_A[x,y])=\\
&f(D_x(a))\cdot_B g(y) + f(a) \cdot_Bg([x,y])
=D_{g(x)}(f(a))\cdot_B g(y) + f(a) \cdot_B[g(x),g(y)]\;.
\end{align*}
The second part is a consequence of (\ref{LR-morph}).
\end{proof}
\subsection{The Exterior Algebra}
For any $n\in\N$, let $\otimes^n_A\mathfrak{g}$ be the $n$-fold 
tensor product of the $A$-module $\mathfrak{g}$ 
with $\otimes^0_A\mathfrak{g}:=A$. Since $A$ is commutative, 
$\otimes^n_A\mathfrak{g}$ is an $A$-module and we write 
$a\cdot_Ax$ for the $A$-scalar multiplication of any $a\in A$ and 
$x\in\bwedge\mathfrak{g}_A$. Note that in general any tensor can be 
expressed in terms of vectors:
\begin{prop}\label{simple_tensor} 
Let $A$ be an associative and commutative algebra with unit, 
$M$ an $A$-module and $\otimes^n_AM$ the appropriate $n$-fold tensor product. 
Then any $x\in \otimes^n_AM$ is an $A$-linear combination of 
\textbf{simple tensors}, i.e. there is a finite index set $I$, tensors
$x_{i,1}\otimes_A\cdots\otimes_A x_{i,n}\in \otimes^n_AM$ and
scalars $a_i\in A$, such that 
\begin{equation}
x=\textstyle\sum_{i\in I}a_i \cdot_A x_{i,1}\otimes_A\cdots\otimes_A x_{i,n}\;.
\end{equation}
\end{prop}
\begin{proof}See for example (\cite{MH}).
\end{proof}
We call such a sum an $A$-\textbf{linear combination}. However in general these 
$A$-linear combinations are not unique. 

Back on Lie Rinehart pairs $(A,\mathfrak{g})$, 
the \textbf{tensor algebra} of the $A$-module $\mathfrak{g}$ is the 
direct sum of all $n$-fold $A$-tensor products
$$
T_A\mathfrak{g}:=\textstyle\bigoplus_{n=0}^\I\bigotimes^n_A\mathfrak{g}\;,
$$
together with an associative but not commutative multiplication given by 
concatenation of tensors 
$\otimes_A: T_A\mathfrak{g} \times T_A\mathfrak{g} \to
T_A\mathfrak{g}\;;\; (x,y)\mapsto x\otimes_A y$. 
This product has a unit $1_A\in \otimes^0_A\mathfrak{g}\simeq A$.

As usual we get the exterior power as the quotient of the tensor power
and the submodule generated by all simples tensors with 'repeated vector products':
\begin{definition}[Exterior Algebra] For any Lie Rinehart pair $(A,\mathfrak{g})$ 
and $n\in\NN$, let  $\bwedge^n_A\mathfrak{g}:=\otimes^n_A\mathfrak{g}/J^n$ be
the quotient module of the $n$-th tensor product and the submodule $J^n$, 
spanned by all $x_1\otimes\cdots\otimes x_n$ with 
$x_i = x_j$ for some $i = j$. Then the direct sum 
\begin{equation}
\bwedge\mathfrak{g}_A:=\textstyle\bigoplus_{n=0}^\I\bwedge^n_A\mathfrak{g}
\end{equation}
together with the quotient 
$\smwedge: \bwedge\mathfrak{g}_A \times \bwedge\mathfrak{g}_A \mapsto
\bwedge\mathfrak{g}_A\;;\; (x,y)\mapsto x\smwedge y$ of the
$A$-tensor multiplication, is called the \textbf{exterior algebra} of 
$(A,\mathfrak{g})$ and the product is called the \textbf{exterior product}. 
\end{definition} 
We write $x_1\smwedge\cdots\smwedge x_n\in\bwedge^n_A\mathfrak{g}$ for the
coset of any tensor $x_1\otimes\cdots\otimes x_n\in \otimes^n_A\mathfrak{g}$ and
in particular $\bwedge^0_A\mathfrak{g}\simeq A$ and 
$\bwedge^1_A\mathfrak{g}\simeq \mathfrak{g}$, since $J^0=\{0\}$ and $J^1=\{0\}$.
If at least one factor in an exterior product is of 
tensor degree zero, we sometimes write $a\cdot_A x$ instead of $a\smwedge x$, 
to stress that the exterior product is just $A$-scalar multiplication in 
that case.
\begin{example}
If $(C^\I(M),\mathfrak{X}(M))$ is the Lie Rinehart pair of smooth 
functions and vector field, the exterior algebra 
$\bwedge\mathfrak{X}(M)_{C^\I(M)}$ is the algebra of
\textbf{multivector fields}.
\end{example}
The exterior algebra is an $A$-module and 
any exterior tensor can be written as a sum (not just a linear combination) of 
simple exterior tensors:
\begin{prop}Let $(A,\mathfrak{g})$ be a Lie Rinehart pair and 
$\bwedge\mathfrak{g}_A$ its exterior algebra. Then
$\bwedge\mathfrak{g}_A$ is an $A$-module and any tensor 
$x\in \bwedge\mathfrak{g}_A$ is an $A$-linear combination of simple exterior products. 
In particular there is a finite index set $I$, scalars $a_i\in A$ and 
simple exterior tensors 
$x_{i,1}\smwedge\cdots\smwedge x_{i,n_i}\in\bwedge\mathfrak{g}_A$, such that
\begin{equation}
x= \textstyle\sum_{i\in I}a_i\cdot x_{i,1}\smwedge\cdots\smwedge x_{i,n_i}\;.
\end{equation}
Moreover any tensor $x\in\bwedge\mathfrak{g}_A$ is a finite sum 
(not just a linear combination) of simple tensors, that is there is a finite index set $I$ and simple exterior tensors 
$x_{i,1}\smwedge\cdots\smwedge x_{i,n_i}\in\bwedge\mathfrak{g}_A$, such that
\begin{equation}
x= \textstyle\sum_{i\in I}x_{i,1}\smwedge\cdots\smwedge x_{i,n_i}\;.
\end{equation}
\end{prop}
\begin{proof}Since $T_A\mathfrak{g}$ is an $A$-module, so is $\bwedge\mathfrak{g}_A$. 
For the first equation, observe that
by prop (\ref{simple_tensor}) 
the $A$-tensor algebra $T_A\mathfrak{g}$ is spanned by simple
tensors $x_1\otimes_A\cdots\otimes_A x_n$. It follows that the quotient 
$\bwedge\mathfrak{g}_A$ is spanned by the appropriate cosets 
$x_1\smwedge\cdots\smwedge x_n$.

To see the second equation apply 
$a\cdot x_1\smwedge\cdots\smwedge x_n = 
(a\cdot_A x_1)\smwedge\cdots\smwedge x_n$ to the first one.
\end{proof}
Morphisms of Lie Rinehart pairs prolong naturally to morphisms of
exterior algebras defined as the direct sum of the scalar part and 
the exterior tensor power of the Lie algebra part. The compatibility 
conditions (\ref{LR-morph}) then guarantee that this map is well defined 
as a morphism of exterior algebras over different scalars.
\begin{definition}(Associated Morphism)
Let $(f,g):(A,\mathfrak{g})\to (B,\mathfrak{h})$ be a morphism of Lie Rinehart pairs. 
The map 
\begin{equation}
\bwedge g_f:\bwedge\mathfrak{g}_A\to\bwedge\mathfrak{h}_B
\end{equation} defined 
on scalars $a\in\bbwedge{0}\mathfrak{g}_A$ by
$f(a)$ and on simple tensors 
$x_1\smwedge_A\cdots\smwedge_A x_n\in\bbwedge{n}\mathfrak{g}_A$ by
$g(x_1)\smwedge_B\cdots\smwedge_B g(x_n)$ and then extended to all of
$\bwedge\mathfrak{g}_A$ by $A$-additivity, is called 
the \textbf{associated morphism} of $(f,g)$.
\end{definition} 
The following proposition shows that associated morphisms are well defined
as morphisms of exterior algebras over different scalars and that the construction
is natural:
\begin{prop}
Let $(f,g):(A,\mathfrak{g})\to (B,\mathfrak{h})$ be a morphism of Lie Rinehart pairs.
Its associated morphism $\bwedge g_f$ is a well defined morphism of
exterior algebras over modules of different rings and in particular the equations
$$
\begin{array}{ccc}
\bwedge g_f((a\cdot_A x)\ssmwedge_A y)= \bwedge g_f(x\ssmwedge_A (a\cdot_A y)) & and &
\bwedge g_f(x \ssmwedge_A y)= \bwedge g_f(x)\ssmwedge_B \bwedge g_f(y)
\end{array}
$$ are satisfied for all $a\in A$ and $x,y\in\bwedge\mathfrak{g}_A$.
If $(h,i):(B,\mathfrak{h})\to (C,\mathfrak{i})$ is another morphism of 
Lie Rinehart pairs then 
$$
\bwedge i_h \circ \bwedge g_f = \bwedge (i\circ g)_{h\circ f}\;.
$$
\end{prop}
\begin{proof}Follows from (\ref{LR-morph}), since the exterior product 
$\bwedge g$ is natural.
\end{proof}
Exterior algebras are equipped with a 
$\Z$-grading coming from the tensor degree, which we need
in the definition of a \textit{graded symmetric} Lie $\I$-algebra later on.
In addition we need a 'reduced' grading to understand the symmetry
of the traditional Schouten-Nijenhuis bracket: 
\begin{definition}(Gradings)
Let $(A,\mathfrak{g})$ be a Lie Rinehart pair and $\bwedge\mathfrak{g}_A$ its
exterior algebra. An element $x\in\bwedge^n_A\mathfrak{g}$ is called
\textbf{homogeneous} and the integer $n$ is called the
\textbf{tensor} degree of $x$, written as 
\begin{equation}
|x|:=n\;.
\end{equation}
In addition the \textbf{antisymmetric} degree of
a homogeneous element $x\in\bbwedge{n}\mathfrak{g}_A$ is the tensor degree, but reduced by one and written as
\begin{equation}\label{antisymm_grading}
deg(x):=n-1\;.
\end{equation}
\end{definition}
Calling the reduced grading 'antisymmetric', will become clear in the next
section, since the traditional Schouten-Nijenhuis bracket behaves 
antisymmetric, with respect to this grading. The reader familiar with 
$\Z$-graded abelian groups will further notice, that the term is the 
correct one, related to the transition between \textit{symmetric} and 
\textit{antisymmetric} \cite{FM}.

With respect to the tensor grading the exterior algebra is concentrated
in non negative degrees and tensors of degree zero are 
precisely the scalars $a\in A$.

Note that the exterior product is mostly seen as some kind of antisymmetric operation,
but since 
\begin{equation}
x\smwedge y = (-1)^{|x||y|}y\smwedge x
\end{equation}
holds on homogeneous tensors, 
it is in fact a \textbf{graded symmetric} product with respect to the tensor grading,
while it has no symmetry at all with respect to the antisymmetric grading.
Moreover it is a graded bilinear map, homogeneous of degree zero with
respect to the tensor grading and homogeneous of degree one with respect to
the antisymmetric grading.
\subsection{The nonexistence of a natural codifferential}
Every graded Lie algebra comes with a (co)differential 
on its reduced symmetric tensor coalgebra \cite{MM} and
considering an ordinary Lie algebra as $\Z$-graded
but concentrated in degree one only, we get a (co)differential on its
exterior power.

Taking this into account, one would guess that such a 
(co)differential is defined for any Lie Rinehart pair, but as we will
see the technique does not apply here anymore. In fact 
on the archetypical example of smooth functions and
vector fields, we can not naturally define a (co)differential
on the exterior algebra other than the zero operator.

This is not a contradiction to the previous mentioned coalgebraic 
approach, since we have to deal with the additional $A$-module structure,
which happens to be trivial on the Lie Rinehart pair $(\R,\mathfrak{g})$ of a Lie algebra.

The following theorem should be seen as a counterexample, giving a Lie Rinehart
pair, without a non trivial (co)differential on its exterior algebra:
\newpage
\begin{theorem}
Let $M$ be a Hausdorff, separable, finite dimensional and
differentiable manifold, $(C^\I(M),\mathfrak{X}(M))$ the Lie Rinehart pair of smooth
functions and vector fields on $M$ and $\bwedge\mathfrak{X}(M)_{C^\I(M)}$
its exterior algebra of \textit{multivector fields}. Then the only
naturally defined (co)differential 
$$
d:\bwedge{}\mathfrak{X}(M)_{C^\I(M)}\to \bwedge{}\mathfrak{X}(M)_{C^\I(M)},
$$
homogeneous of degree $-1$ with respect to the tensor grading, is the zero operator.
\end{theorem}
\begin{proof} We look at a slightly more general situation and proof that 
for any $n\in\N$ there is no naturally defined map 
$d:\bbwedge{n}\mathfrak{X}(M)_{C^\I(M)}\to \bbwedge{n-1}\mathfrak{X}(M)_{C^\I(M)}$
other than the zero map at all. 
Since the codifferential is homogeneous of degree $-1$, this 
will include the situation of the theorem.

As we require this map to be a \textit{natural operator}, we
can use the technique of \cite{KMS}. According to theorem 14.18 in \cite{KMS}, 
those operators of order $r$ are in one to one correspondence with 
maps $f:T^{r+1}_m(\bbwedge{n}\R^m)\to \bbwedge{n-1}\R^m$, equivariant with
respect to the associated action of the jet group $G_m^{r+1}$ on domain and
codomain for any $m\in\N$.
From \cite{KMS} prop 14.20 we know further
$$
T^r_m(\bbwedge{n}\R^m)\simeq 
(\bbwedge{n}\R^m)\oplus(\bbwedge{n}\R^m\otimes\R^{m*})\oplus\cdots\oplus
(\bbwedge{n}\R^m\otimes S^{r}\R^{m*})
$$
and write 
$x^{j_1,\ldots,j_n}$, $x^{j_1\ldots,j_n,}{}_{j_{n+1}}$,
... , $x^{j_1,\ldots,j_n,}{}_{j_{n+1},\ldots,j_{n+r}}$ for the 
local coordinates of $T^r_m(\bbwedge{n}\R^m)$, that are the prolongations of the 
canonical coordinates $x^j$ in $\R^m$. In particular any such expression is
antisymmetric in all upper and symmetric in all lower indices.

The required actions $l^{r+1}$ of $G^{r+1}_m$ on $T^r_m(\bbwedge{n}\R^m)$ are pretty
involved, but fortunately we only need the actions of the general linear
group $Gl(m)$, seen as a subset of the $(r+1)$-th order jet group 
$G^{r+1}_m$. According to \cite{KMS} proposition 14.20 this restricted action is 
purely tensorial and given by
$\tilde{x}^{j_1,\ldots,j_n}= x^{i_1,\ldots,i_n}
 a^{j_1}{}_{i_1}\cdots a^{j_n}{}_{i_n}$,
$\tilde{x}^{j_1,\ldots,j_n,}{}_{j_{n+1}}= 
x^{i_1,\ldots,i_n,}{}_{i_{n+1}}a^{j_1}{}_{i_1}\cdots 
 a^{j_n}{}_{i_n}a^{i_{n+1}}{}_{j_{n+1}}
$ $,\ldots,$\\
$
\tilde{x}^{j_1,\ldots,j_n,}{}_{j_{n+1},\ldots, j_{n+r}}= 
x^{i_1,\ldots,i_n,}{}_{i_{n+1},\ldots,i_{n+r}}
 a^{j_n}{}_{i_n}\cdots a^{j_n}{}_{i_n}a^{i_{n+1}}{}_{j_{n+1}}\cdots 
  a^{i_{n+r}}{}_{j_{n+r}}
$
and the restricted action on $\bwedge^{n-1}\R^m$ is tensorial, too.

Now choose some real $\lambda>0$ and consider the general linear 
transformations $a^j{}_i\in Gl(m)$, defined by 
$a^i{}_i=\lambda$ and $a^j{}_i=0$ for $i\neq j$. The action on 
$\bwedge^{n-1}\R^m$ and $T^r_m(\bbwedge{n}\R^m)$ is  particularly easy to compute and given by
$l^{r+1}(a^j{}_i;x^{j_1,\ldots,j_n})=\lambda^{n-1}x^{j_1,\ldots,j_n}$ and
\begin{multline*}
l^{r+1}(a^j{}_i;x^{j_1,\ldots,j_n},x^{j_1,\ldots,j_n,}{}_{j_{n+1}},\ldots,
x^{j_1,\ldots,j_n,}{}_{j_{n+1},\ldots,j_{n+r}})=\\
\left(\lambda^nx^{j_1,\ldots,j_n},\lambda^{n+1} x^{j_1,\ldots,j_n,}{}_{j_{n+1}},\ldots,
\lambda^{n+r} x^{j_1,\ldots,j_n,}{}_{j_{n+1},\ldots,j_{n+r}}\right)\;.
\end{multline*}
It follows that any map $f:T^r_m(\bbwedge{n}\R^m)\to \bbwedge{n-1}\R^m$,
equivariant with respect to the action $l^{r+1}$
has to satisfy the homogenity condition
\begin{multline*}
\lambda^{n-1}f\left(x^{j_1,\ldots,j_n},x^{j_1,\ldots,j_n,}{}_{j_{n+1}},\ldots,
x^{j_1,\ldots,j_n,}{}_{j_{n+1},\ldots,j_{n+r}}\right)=\\
f\left(\lambda^n x^{j_1,\ldots,j_n},\lambda^{n+1} x^{j_1,\ldots,j_n,}{}_{j_{n+1}},
\ldots,\lambda^{n+r} x^{j_1,\ldots,j_n,}{}_{j_{n+1},\ldots,j_{n+r}}\right)
\end{multline*}
for all real $\lambda>0$, but by the homogeneous function theorem \cite{KMS} (24.1)
such a map (other than the zero morphism) only exists if the equation
$$n d_1 + (n+1)d_2 \cdots + (n+r)d_r = n-1$$ has solutions 
$d_1,\ldots,d_r\in \N$, which it hasn't. It follows that the zero map is the
only equivariant function and consequently the only natural operator is the zero operator. 
\end{proof}
In \cite{GR} Rinehart defined a (co)differential for any Lie Rinehart pair,
but on the tensor product of the exterior algebra and the universal enveloping
algebra of the Lie algebra instead. In case of smooth functions and 
vector fields this gives a structure dual to the usual De Rham complex of
differential forms. 
\begin{remark} On the Lie Rinehart pair $(\R,\mathfrak{g})$ of 
a Lie algebra with its trivial $\mathfrak{g}$-module structure
on $\R$, a (co)differential
$d:\bwedge\mathfrak{g} \to \bwedge\mathfrak{g}$
is defined by $d(\lambda)=0$ on scalars $\lambda\in \R$ as well as $d(x)=0$
on vectors $x\in\mathfrak{g}$ and by
$$d(x_1\smwedge\cdots\smwedge x_n):=
\textstyle\sum_{s\in Sh(2,n-2)}e(s)\,
  [x_{s(1)},x_{s(2)}]\smwedge x_{s(3)}\smwedge\cdots\smwedge x_{s(n)}   
$$
on simple tensors
$x_1\smwedge\cdots\smwedge x_n\in \bwedge \mathfrak{g}$
and is then extend to 
all of $\bwedge\mathfrak{g}$ by linearity. Except for the degree
zero part (which is trivial) 
this is the 'coalgebraic' (co)differential as it appears for example in \cite{MM}.
\end{remark}
If we try to define a similar map on the exterior power of an
arbitrary Lie Rinehart pair, the operator is not necessarily well defined
with respect to the additional module structure and we could face situations
like 
$$
d((a\cdot x_1)\smwedge x_2)\neq
d(x_1\smwedge(a\cdot x_2))\;.
$$
\subsection{The Schouten-Nijenhuis bracket}
The Lie bracket on vector fields can be extended to a
graded Lie bracket on \textit{multivector fields}, usually called 
Schouten-Nijenhuis bracket \cite{CM},\cite{PM}. We show that this can be generalized 
verbatim to arbitrary Lie Rinehart pairs.
\begin{definition}[Schouten-Nijenhuis bracket]
Let $(A,\mathfrak{g})$ be a Lie Rinehart pair and $\bwedge\mathfrak{g}_A$ its
exterior algebra. The map
\begin{equation}
\left[\cdot\;,\cdot\right]: \bwedge\mathfrak{g}_A \times \bwedge\mathfrak{g}_A \to
 \bwedge\mathfrak{g}_A\;,
\end{equation}
defined by $[a,b]=0$ as well as $[x,a]=[a,x]=D_x(a)$ on scalars $a,b\in A$ and 
vectors $x\in\mathfrak{g}$ and by
\begin{multline}\label{SN_1}
[x_1\smwedge\cdots \smwedge x_n,y_1\smwedge\cdots\smwedge y_m]=\\
	\textstyle\sum_{i,j}(-1)^{i+j}[x_i,y_j]\smwedge x_1\smwedge\cdots\smwedge 
		\widehat{x_i}\smwedge\cdots\smwedge x_n\smwedge
			y_1\smwedge\cdots\smwedge \widehat{y_j}\smwedge\cdots\smwedge y_m
\end{multline}
on simple tensors 
$x_{1}\smwedge\cdots\smwedge x_{n}$, 
$y_{1}\smwedge\cdots\smwedge y_{m}\in\bwedge\mathfrak{g}_A$ 
and then extend to all of $\bwedge\mathfrak{g}_A$ by $A$-additivity, is called the
\textbf{(antisymmetric) Schouten Nijenhuis bracket} of $(A,\mathfrak{g})$.
\end{definition}
This is the traditional definition as it appears for example in \cite{CM}. 
In case of simple tensors an equivalent but more symmetric expression 
for (\ref{SN_1}) is given by
$$\label{SN_2}
\textstyle\sum_{s\in Sh(1,n-1)}\sum_{t\in Sh(1,m-1)}
 e(s)e(t)[x_{s_1},y_{t_1}]\smwedge x_{s_2}
  \smwedge\cdots\smwedge x_{s_n}\smwedge y_{t_2}\smwedge\cdots\smwedge y_{t_m}\;.
$$
Care has to be taken, to get the symmetry of the Schouten-Nijenhuis bracket right. 
In fact we have to consider the antisymmetric grading (\ref{antisymm_grading}) of 
$\bwedge\mathfrak{g}_A$ to understand it properly. With this grading 
the common commutation equation
$$
[x,y]=-(-1)^{(|x|-1)(|y|-1)}[y,x]
$$
suddenly becomes more conceptual and just says that the Schouten-Nijenhuis bracket is 
\textbf{graded antisymmetric} with respect to the \textit{antisymmetric grading}.
Later we have to deal with a graded symmetric incarnation of the bracket and
that's why we call this one the \textit{antisymmetric} bracket.

Proofing its properties has been done in the situation of multivector fields
at many places before \cite{CM}, \cite{PM} and we only recapitulate the basic facts
for completeness:
\begin{theorem}Let $(A,\mathfrak{g})$ be a Lie Rinehart pair with
exterior algebra $\bwedge\mathfrak{g}_A$. The 
Schouten-Nijenhuis bracket $[.,.]$ is a $\R$-bilinear, graded antisymmetric operator,
homogeneous of degree zero with respect to the
antisymmetric grading and in particular the equation
$$
\begin{array}{cc}
[x,y]=-(-1)^{deg(x)deg(y)}[y,x],&
[x,y\smwedge z] = [x,y]\smwedge z + (-1)^{deg(x)(deg(y)-1)}y\smwedge [x,z]
\end{array}
$$
as well as the graded Jacobi equation in its antisymmetric incarnation
$$
(-1)^{deg(x)deg(z)}[x,[y,z]]+(-1)^{deg(x)deg(y)}[y,[z,x]]
+(-1)^{deg(y)deg(z)}[z,[x,y]] = 0
$$
are satisfied for any homogeneous tensors $x,y,z\in\bwedge\mathfrak{g}_A$.
\end{theorem}
\begin{proof}On homogeneous tensors $x$ and $y$ we get 
$|[x,y]|=|x|+|y|-1$ and the bracket is homogeneous of tensor 
degree $-1$. This in turns gives $deg([x,y])=|x|+|y|-2=deg(x)+deg(y)$ and consequently
the bracket is homogeneous of degree zero with respect to the antisymmetric
grading.

All other properties are computed verbatim as for the Schouten-Nijenhuis
bracket of multivector fields.
\end{proof}
Now lets look at exactly the same situation, but from the tensor grading of
$\bwedge\mathfrak{g}_A$. There is a 'natural transformation' between the tensor
and the antisymmetric grading 
called \textit{decalag\'e morphism} \cite{FM}, which can be used to
transform graded antisymmetric operators into graded symmetric ones and
vis versa. Applied to the antisymmetric Schouten-Nijenhuis 
bracket, this gives a graded \textit{symmetric} operator:
\begin{definition}Let $(A,\mathfrak{g})$ be a Lie Rinehart pair and 
$[\cdot,\cdot]$ the antisymmetric Schouten-Nijenuis bracket on $\bwedge\mathfrak{g}_A$. 
Then the operator
\begin{equation}
\{\cdot,\cdot\}: \bwedge\mathfrak{g}\times\bwedge\mathfrak{g}\to
 \bwedge\mathfrak{g}
\end{equation}
defined by $\{x,y\}=e(x)[y,x]$
on homogeneous tensors $x,y\in\bwedge\mathfrak{g}$ and extended to all of 
$\bwedge\mathfrak{g}$ by $A$-additivity is called the 
\textbf{symmetric Schouten-Nijenhuis bracket}.
\end{definition} 
All properties of the antisymmetric bracket transform properly under the
decalgn\'e morphism and both brackets coincides with the original Lie bracket
on vectors:
\begin{corollary}The symmetric Schouten-Nijenhuis bracket $\{.,.\}$ 
is a $\R$-bilinear, graded symmetric operator,
homogeneous of degree $-1$ with respect to the
tensor grading. In particular the symmetry equation as well as the Jacobi equation
$$\label{symmetric_jacobi}
\begin{array}{cc}
\{x,y\}=e(x,y)\{y,x\}, &
\textstyle\sum_{s\in Sh(2,1)}e\left(s;x_1,x_2,x_3\right)
\{\{x_{s(1)},x_{s(2)}\},x_{s(3)}\}=0
\end{array}$$
is satisfied for all homogeneous tensors $x,y,z\in\bwedge\mathfrak{g}_A$
and on vectors $x,y\in\bbwedge{1}\mathfrak{g}$, the bracket equals the original
Lie bracket of $\mathfrak{g}$ that is
$\{x,y\}=[x,y]$.
\end{corollary}
\begin{proof}This follows from the properties of the decalgn\'e morphism or
can else be verified by simple computations.
\end{proof}
According to a better readable text we will always use the Koszuls sign 
conventions (\ref{Koszul_convention}), when it comes to expressions, which
are graded symmetric with respect to the tensor grading.

Any morphism of Lie Rinehart pairs prolongs
to a morphism of exterior algebras and the following corollary
shows that this is in fact a morphism of graded Lie algebras with respect
to the Schouten-Nijenuis bracket:
\begin{prop}\label{schouten-morphism}
Let $(f,g):(A,\mathfrak{g})\to (B,\mathfrak{h})$ be a morphism of Lie Rinehart pairs. 
The associated morphism $\bwedge g_f:\bwedge\mathfrak{g}_A\to\bwedge\mathfrak{h}_B$ of
exterior algebras is a morphism of graded Lie algebras,
with respect to the Schouten-Nijenuis bracket. 
\end{prop}
\begin{proof}$\bwedge g_f$ is homogeneous of degree zero with respect to the
tensor grading. The rest follows, since $\bwedge g$ is natural and $g$ a 
Lie algebra morphism.
\end{proof}
 \begin{remark}
Since $\bwedge^0\mathfrak{g}\simeq A$ and $\bwedge^1\mathfrak{g}_A\simeq \mathfrak{g}$,
there is a natural injection 
$A\oplus\mathfrak{g}\hookrightarrow\bwedge\mathfrak{g}_A$ of graded vector spaces 
and since the Schouten-Nijenhuis bracket coincides with the bracket of 
$A\oplus\mathfrak{g}$
on scalars and vectors, this is in fact a morphism of graded Lie algebras.
\end{remark}
\section{The Lie $\I$-Algebra of a Lie Rinehart pair}We expand the
Schouten-Nijenhuis algebra into a 
Lie $\I$-algebra with non trivial higher brackets. Since the zero
morphism is the only general (co)differential in this setting, 
the 'unary' bracket has to vanish.

The structure we obtain is particularly simple and merely a change 
in perspective. Its real advantage lies in the fact, 
that we gain access to a lot more morphisms. 
Morphisms which are just not there, when we restrict to the Schouten-Nijenhuis
picture.

Finally we show that there is a weak injection of any Lie Rinehart pair into its
associated Lie $\I$-algebra.
\begin{definition}[Higher Lie Brackets]Let $(A,\mathfrak{g})$ be a Lie
Rinehart pair with exterior algebra $\bwedge\mathfrak{g}_A$ and 
$[\cdot,\cdot]$ the antisymmetric Schouten-Nijenhuis bracket. Then the 
\textbf{Lie n-bracket}
\begin{equation}
\{\cdot,\cdots,\cdot\}_n : \bwedge\mathfrak{g}_A \times \cdots \times
 \bwedge\mathfrak{g}_A \to \bwedge\mathfrak{g}_A
\end{equation}
is defined for any integer $n \geq 2$ and homogeneous tensors
$x_1,\ldots,x_n\in \bwedge\mathfrak{g}_A$ by
\begin{multline*}
\{x_1,\ldots,x_n\}_n:=\\ \textstyle\sum_{s\in Sh(2,n-2)}e(s;x_1,\ldots,x_n)e(x_{s(1)})\;
x_{s(n)}\smwedge\cdots\smwedge x_{s(3)}\smwedge[x_{s(2)},x_{s(1)}]
\end{multline*}
and then extended to all of $\bwedge\mathfrak{g}_A$ by $A$-additivity.
\end{definition}
In particular the Lie $2$-bracket is just the symmetric Schouten-Nijenhuis
bracket. If we referee to the Lie $n$-bracket for any $n \in \N$, we
consider the zero operator as the 'Lie 1-bracket',
that is we write $\{\;\cdot\;\}_1$ for the zero operator
in this context.

The following proposition provides the technical details to
show that the sequence of Lie $n$-brackets defines a 
non negatively graded Lie $\I$-algebra on the exterior power of any Lie 
Rinehart pair.
\begin{prop}The Lie $n$-bracket $\{\cdot,\cdots,\cdot\}_n$ 
is a graded symmetric, $n$-linear operator, homogeneous of 
tensor degree $-1$ for any $n\in\N$ and for $p$, $q\in\N$ with $p+q=n+1$ 
and $p>1$ as well as $q>1$ the equation 
\begin{equation}\label{jacobi_like}
\textstyle\sum_{s\in Sh(q,p-1)}e(s;x_1,\ldots,x_n)
 \{\{x_{s(1)},\ldots,x_{s(q)}\}_q,x_{s(q+1)},\ldots,x_{s(n)}\}_p=0
\end{equation}
is satisfied.
\end{prop}
\begin{proof}
The exterior product is homogeneous of degree zero and the 
antisymmetric Schouten-Nijenhuis bracket is 
homogeneous of degree $-1$,
with respect to the tensor grading. It follows that any $n$-ary bracket 
is homogeneous of degree $-1$.

To proof graded symmetry observe that the exterior product is graded
symmetric with respect to the tensor grading and that the expression
$e(x_1)[x_2,x_1]$ is precisely the graded symmetric Schouten-Nijenhuis bracket. 
The symmetry of the Lie $n$-bracket then follows since it is a graded symmetric 
composition of both operators.

Now to see that any of the 'Jacobi-like' shuffle sums
$$
\textstyle\sum_{s\in Sh\left(q,p-1\right)}e\left(s;x_1,\ldots,x_n\right)
 \{\{x_{s(1)},\ldots,x_{s(q)}\}_q,x_{s(q+1)},\ldots,x_{s(n)}\}_p
$$
vanishes, observe that for $n=3$ and $p=q=2$ this is nothing but the
ordinary (graded symmetric) Jacobi expression (\ref{symmetric_jacobi}). 
To see it for $n\geq 4$, 
let $x_1,\ldots,x_n\in\bwedge\mathfrak{g}_A$ be homogeneous and 
$p,q\in\N$ with $p+q=n+1$ and $p,q\geq 2$. 
First use the graded symmetry of the brackets, to rewrite the shuffle sum
into a sum over arbitrary permutations. This gives
$$\textstyle\frac{1}{q!\left(p-1\right)!}\sum_{s\in S_n}
 e\left(s;x_1,\ldots,x_n\right)
  \{\{x_{s(1)},\ldots,x_{s(q)}\}_q,x_{s(q+1)},\ldots,x_{s(n)}\}_p\;.
$$
Then apply the definition of $\{\cdot,\ldots,\cdot\}_q$.
Writing $Sh_{\{s(1),\ldots,s(q)\}}\left(i,j\right)$ for
the set of all $(i,j)$-shuffles but explicit as permutations of the set $\{s(1),\ldots,s(q)\}$ 
the previous expression becomes
\begin{multline*}
\textstyle\frac{1}{q!\left(p-1\right)!}\sum_{s\in S_n}
 \sum_{t\in Sh_{\{s(1),\ldots,s(q)\}}\left(2,q-2\right)}
 e\left(s;x_1,\ldots,x_n\right)\,e\left(t;x_{s(1)},\ldots,x_{s(q)}\right)\cdot \\
\phantom{=.}\cdot e\left(x_{ts(1)}\right)\{x_{ts(q)}\smwedge\cdots\smwedge x_{ts(3)}\smwedge[\,x_{ts(2)},x_{ts(1)}\,],
 x_{s(q+1)},\ldots,x_{s(n)}\}_p\;.
\end{multline*}
Now observe, that for any $s\in S_n$  
and shuffle $t\in Sh_{\{s(1),\ldots,s(q)\}}(2,q-2)$, 
the permutation $(ts(1),\ldots,ts(q),s(q+1),\ldots,s(n))$ is again an element of
$S_n$ and since there are precisely $\frac{q!}{(q-2)!2!}$ many
shuffles in $Sh_{\{s(1),\ldots,s(q)\}}(2,q-2)$ we can just 'absorb' the second sum
over shuffles in the previous expression into the sum over general permutation. After reindexing we get: 
\begin{multline*}
\textstyle\frac{1}{q!(p-1)!}\frac{q!}{(q-2)!2!}\sum_{s\in S_n}
 e\left(s;x_1,\ldots,x_n\right)e\left(x_{s(1)}\right)\cdot\\
\phantom{=.}\cdot
 \{x_{s(q)}\smwedge\cdots\smwedge x_{s(3)}\smwedge[\,x_{s(2)},x_{s(1)}\,],
  x_{s(q+1)},\ldots,x_{s(n)}\}_p
\end{multline*}
Then apply the definition of the bracket a second time. 
Care has to be taken regarding the first element.
In fact there are two possible positions for it:
At the most right position, i.e as the second argument inside of the Schouten-Nijenhuis bracket or at the most right position outside of the bracket. Taking
this into account split the expression into two parts according to the position
of the first element. If $p=2$ the second shuffle sum is omitted:
\begin{align*}  
&\textstyle\frac{1}{2(p-1)!(q-2)!}\sum_{s\in S_n}
 \sum_{t\in Sh_{\{s(q+1),\ldots,s(n)\}}(1,p-2)} e\left(s;x_1,\ldots,x_n\right)\cdot\\ 
&\phantom{=.}\cdot e\left(t;x_{s(q+1)},\ldots,x_{s(n)}\right)e\left(x_{s(1)}\right)
 e\left(x_{s(q)}\smwedge\cdots\smwedge
  x_{s(3)}\smwedge[\,x_{s(2)},x_{s(1)}\,]\,\right)\cdot\\
&\phantom{=.}\cdot x_{ts(n)}\smwedge\cdots\smwedge x_{ts(q+2)}\smwedge
  [x_{ts(q+1)},x_{s(q)}\smwedge\cdots\smwedge x_{s(3)}\smwedge[x_{s(2)},x_{s(1)}]]\\ 
&+\textstyle\frac{1}{2\left(p-1\right)!\left(q-2\right)!}\sum_{s\in S_n}
 \sum_{t\in Sh_{\{s(q+1),\ldots,s(n)\}}\left(2,p-3\right)} 
  e\left(s;x_1,\ldots,x_n\right)\,\cdot\\
&\phantom{=.}\cdot e\left(t;x_{s(q+1)},\ldots,x_{s(n)}\right)
 e\left(x_{s(q)}\smwedge\cdots\smwedge x_{s(3)}\smwedge
  [x_{s(2)},x_{s(1)}],x_{ts(q+1)}\right)\,\cdot\\
&\phantom{=.}\cdot 
e\left(x_{s(q)}\smwedge\cdots\smwedge x_{s(3)}\smwedge
 [x_{s(2)},x_{s(1)}],x_{ts(q+2)}\right)e\left(x_{s(1)}\right)
  e\left(x_{ts(q+1)}\right)\,\cdot\\
&\phantom{=.}\cdot 
 x_{ts(n)}\smwedge\cdots\smwedge x_{ts(q+3)}\smwedge
  x_{s(q)}\smwedge\cdots\smwedge x_{s(3)}\smwedge[x_{s(2)},x_{s(1)}]\smwedge
   [x_{ts(q+2)},x_{ts(q+1)}]
\end{align*}
Again we 'absorb' the additional shuffle sums in both cases into the
appropriate sum over general permutations. After reindexing and simplification
this becomes
\begin{align} 
&-\textstyle\frac{1}{2(q-2)!(p-2)!}\sum_{s\in S_n}
 e(s;x_1,\ldots,x_n)e(x_{s(2)})\cdots e(x_{s(q)})\,
x_{s(n)}\smwedge\cdots\smwedge x_{s(q+2)}\smwedge\nonumber\\
&\phantom{=.}
 \smwedge [\,x_{s(q+1)},x_{s(q)}\smwedge\cdots\smwedge x_{s(3)}\smwedge
  [x_{s(2)},x_{s(1)}\,]\,]\label{eq_19}\\ 
&+\textstyle\frac{1}{4(q-2)!(p-3)!}\sum_{s\in S_n}
 e\left(s;x_1,\ldots,x_n\right)e\left(x_{s(2)}\right)e\left(x_{s(3)}\right)\,
  x_{s(n)}\smwedge\cdots\smwedge x_{s(5)}\smwedge\nonumber\\
&\phantom{=.}
 \smwedge [x_{s(4)},x_{s(3)}]\smwedge [x_{s(2)},x_{s(1)}]\nonumber     
\end{align}
and again the second sum is omitted for $p=2$. Now lets look on both sums separately. From
\begin{align*}
&e(x_2)e(x_3)[x_4,x_3]\smwedge [x_2,x_1]
= e(x_2)e(x_3)e([x_2,x_1],[x_4,x_3])[x_2,x_1]\smwedge [x_4,x_3]\\
&=-e(x_1,x_3)e(x_1,x_4)e(x_2,x_3)e(x_2,x_4)e(x_1)e(x_4)[x_2,x_1]\smwedge [x_4,x_3]
\end{align*}
we see that the expression
\begin{multline}\label{eq_aa}
\textstyle\sum_{s\in S_n}e(s;x_1,\ldots,x_n)e(x_{s(2)})e(x_{s(3)})\;\cdot\\
 \cdot x_{s(n)}\smwedge\cdots\smwedge x_{s(5)}\smwedge[x_{s(4)},x_{s(3)}]\smwedge
  [x_{s(2)},x_{s(1)}]
\end{multline}
vanishes, since for any permutation $s\in S_n$ there is precisely one permutation
$t\in S_n$ with $t=(s(3),s(4),s(1),s(2),s(5),\ldots,s(n))$
and then the term for $t$ cancel against the term for $s$.

Now only the first sum in (\ref{eq_19}) remains, but for $q=2$ it
vanishes too due to the graded Jacobi equation  of the symmetric
Schouten-Nijenhuis bracket. 
This proofs the equation for $q=2$.

For $q\geq 3$ use the Poisson identity 
$[x,y\smwedge z]=[x,y]\smwedge z+e(x,y)e(y)y\smwedge [x,z]$ 
of the antisymmetric Schouten-Nijenhuis bracket and omit the constant
factor in (\ref{eq_19}). After simplification we get
\begin{align}\label{eq_21}
&\textstyle\sum_{s\in S_n}e\left(s;x_1,\ldots,x_n\right)
 e\left(x_{s(2)}\right)\cdots e\left(x_{s(q)}\right)\,\cdot\nonumber\\
&\phantom{=.}\cdot
 x_{s(n)}\smwedge\cdots\smwedge x_{s(q+2)}\smwedge
  [x_{s(q+1)},x_{s(q)}\smwedge\cdots\smwedge x_{s(3)}]
   \smwedge[x_{s(2)},x_{s(1)}]\\
&+\textstyle\sum_{s\in S_n}e\left(s;x_1,\ldots,x_n\right)
 e\left(x_{s(2)}\right)
  x_{s(n)}\smwedge\cdots\smwedge x_{s(4)}\smwedge[x_{s(3)},[x_{s(2)},x_{s(1)}]]
  \nonumber\,.
\end{align}
The second sum vanishes due to the graded Jacobi equation of the 
symmetric Schouten-Nijenhuis bracket. For $q=3$ the first sum vanishes too,
since we arrive at situation (\ref{eq_aa}). This proofs the equation for 
$q=3$.

For $q\geq 4$ we apply the Poisson identity to the remaining part again and 
after simplification we get
\begin{align*} 
&\textstyle\sum_{s\in S_n}e\left(s;x_1,\ldots,x_n\right)
 e\left(x_{s(2)}\right)e\left(x_{s(3)}\right)
 x_{s(n)}\smwedge\cdots\smwedge x_{s(5)}\smwedge
  [x_{s(4)},x_{s(3)}]\smwedge[x_{s(2)},x_{s(1)}]\\
&+\textstyle\sum_{s\in S_n}e\left(s;x_1,\ldots,x_n\right)
 e\left(x_{s(2)}\right)\cdots e\left(x_{s(q-1)}\right)\,\cdot\\
&\phantom{=.}\cdot
 x_{s(n)}\smwedge\cdots\smwedge
  x_{s(q+1)}\smwedge[x_{s(q)},x_{s(q-1)}\smwedge\cdots\smwedge x_{s(3)}] 
   \smwedge[x_{s(2)},x_{s(1)}]
\end{align*}
Again the first sum vanishes since it is the same situation as in (\ref{eq_aa}) 
and the second sum vanishes for $q=4$ due to the same reason. 
For $q>4$ the second sum equals the first in (\ref{eq_21}) but for $q-1$ instead. 
Consequently we have to repeat the last computation 
$(q-4)$-times, to arrive at an expression that is equal to (\ref{eq_aa}) and hence
vanishes. This proofs the equation. 
\end{proof}
Taking into account, that the unary bracket $\{\cdot\}_1$
has to be the zero operator, we can combine
the brackets into a Lie $\I$-algebra on the exterior power of any Lie Rinehart
pair:
\begin{theorem}[The Lie $\I$-algebra]
Let $(A,\mathfrak{g})$ be a Lie Rinehart pair with exterior
power $\bwedge\mathfrak{g}_A$. Then
$\left(\bwedge\mathfrak{g}_A,(\{\cdot,\ldots,\cdot\}_k)_{k\in\N}\right)$
is a Lie $\I$-algebra, concentrated in non negative degrees.
\end{theorem}
\begin{proof}
All operators are graded symmetric and homogeneous of degree $-1$ with respect
to the tensor grading and the weak Jacobi identities (\ref{sh_Jacobi}) 
follow from (\ref{jacobi_like}), since $\{\cdot\}_1$ is the zero operator. 
Moreover the exterior algebra is concentrated in non-negative degrees with
respect to the tensor grading.
\end{proof}
\begin{remark}
Note that the Lie $\I$-structure is particularly simple in this case:
The (co)differential is the zero operator and from (\ref{jacobi_like}) we see, that 
each particular shuffle sum already vanishes for fixed $p$ and $q$  in the
weak Jacobi identities (\ref{sh_Jacobi}).
\end{remark}
The following theorem shows, that the construction is natural with respect to
morphisms of Lie Rinehart pairs. In fact any morphism of Lie Rinehart pairs 
gives rise to a \textit{strict} morphism of Lie $\I$-algebras:
\begin{theorem}
Let $(f,g):(A,\mathfrak{g})\to (B,\mathfrak{h})$ be a morphism of Lie Rinehart pairs. 
The associated exterior algebra morphism 
$\bwedge g_f:\bwedge\mathfrak{g}_A\to\bwedge\mathfrak{h}_B$ is a strict morphism of
Lie $\I$-algebras.
\end{theorem}
\begin{proof}We need to show that $\bwedge g_f$ commutes 
with the Lie $n$-bracket for any $n\in\N$. 
For $n=1$ this is trivial and for $n=2$ this is proposition (\ref{schouten-morphism}). 
For $n\geq 3$ it follows, since $\bwedge g$ commutes with the exterior 
product and the Schouten-Nijenhuis bracket.
\end{proof} 
The natural injection 
$A\oplus\mathfrak{g}\hookrightarrow \bwedge\mathfrak{g}_A\;;\;
(a,x)\mapsto(a,x)$ 
can't be a morphism of Lie $\I$-algebras, 
since it has to commute with all higher brackets, 
but these brackets are zero on $A\oplus\mathfrak{g}$.
However as the following theorem shows, a 
natural injection now comes as a weak morphism of Lie $\I$-algebras: 
\begin{definition}Let $(A,\mathfrak{g})$ be a Lie Rinehart pair with 
associated graded Lie algebra $A\oplus\mathfrak{g}$,
exterior algebra $\bwedge\mathfrak{g}_A$ and 
\begin{equation}
\begin{array}{crcl}
i_n: & A\oplus\mathfrak{g}\times \cdots \times A\oplus\mathfrak{g} &\to &
 \bwedge\mathfrak{g}_A\\
     &   (x_1,\ldots,x_n) & \mapsto &
 (-1)^{n-1}(n-1)\,!\cdot x_n\smwedge\cdots\smwedge x_1
\end{array}
\end{equation}
for any $n\in \N$. Then the sequence $i_\I:=(i_n)_{n\in\N}$ is called the
\textbf{natural injection} of the Lie Rinehart pair into its exterior 
Lie $\I$-algebra.
\end{definition}
The following theorem shown, that this sequence of multilinear maps is in fact
a morphism of Lie $\I$-algebras:
\begin{theorem}The natural injection 
$i_\I: A\oplus\mathfrak{g}\to \bwedge\mathfrak{g}_A$ is a morphism of Lie $\I$-algebras.
\end{theorem}
\begin{proof}Any map $i_{n}$ is graded symmetric and homogeneous of 
tensor degree zero, since the same holds for the exterior product.

In (the Lie $\I$-algebra) $A\oplus\mathfrak{g}$, only the binary bracket 
does not vanish and since $\{\cdot\}_1$ is the zero operator, the general structure equation (\ref{lie-infty-morph}) 
of a Lie $\I$-algebra morphism simplifies for any $n\geq 2$ into
\begin{multline*}
\textstyle\sum_{s\in Sh(2,n-2)}e(s)
 i_{n-1}([x_{s(1)},x_{s(2)}],x_{s(3)},\ldots,x_{s(n)})=\\
\textstyle\sum_{p=2}^n
 \frac{1}{p!}\sum_{s\in Sh(j_{1},\ldots,j_{p})}^{j_{1}+\ldots+j_{p}=n}
  e(s)\{i_{j_{1}}(x_{s(1)},\ldots,x_{s(k_{1})}),...,
   i_{j_{p}}(x_{s(n-j_{p}+1)},\ldots,x_{s(n)})\}_p\;.
\end{multline*}

Now assume $n\geq 2$ and $2\leq p \leq n$ as well as $j_1+\ldots+j_p=n$ for positive
integers $j_k$. We use the graded symmetry of the Lie $n$-brackets and the 
maps $i_n$, 
to rewrite the shuffle sum at the right side of the structure equation into a sum over arbitrary permutations:
\begin{multline*}
\textstyle\sum_{s\in Sh(j_{1},\ldots,j_{p})}e(s)
 \{i_{j_{1}}(x_{s(1)},\ldots,x_{s(j_{1})}),
  \ldots,i_{j_{p}}(x_{s(n-j_{p}+1)},\ldots,x_{s(n)})\}_{p}=\\	
\textstyle\frac{1}{j_{1}!\cdots j_{p}!}\sum_{s\in S_{n}}e(s)
 \{i_{j_{1}}(x_{s(1)},\ldots,x_{s(j_{1})}),
  \ldots,i_{j_{p}}(x_{s(n-j_{p}+1)},\ldots,x_{s(n)})\}_{p}
\end{multline*}

To reorganize this, define $j_0:=0$ and write 
$X_{k}^{s}:=i_{j_{k}}(x_{s(j_{k-1}+1)},...,x_{s(j_{k-1}+j_{k})})$ 
for any given vectors $x_1,\ldots,x_n\in\mathfrak{g}$, permutation $s\in S_n$
and $1\leq k \leq p$. 
Then $|X_{k}^s|=j_{k}$ and we can abbreviate the previous expression into
$$\textstyle\frac{1}{j_{1}!\cdots j_{p}!}
 \sum_{s\in S_{n}}e(s)\{X_{1}^s,\ldots,X_{p}^s\}_{p}\;.$$
Applying the definition of the Lie $p$-bracket
is now straight forward and leads to the expression
\begin{multline*}
\textstyle\frac{1}{j_{1}!\cdots j_{p}!}\sum_{s\in S_{n}}e(s)
 \sum_{t\in Sh(2,p-2)}e(t,X_{1}^{s},\ldots,X_{p}^{s})\cdot\\
  \cdot e(X_{t(1)}^{s})X_{t(p)}^{s}\smwedge\cdots\smwedge 
   X_{t(3)}^{s}\smwedge[X_{t(2)}^{s},X_{t(1)}^{s}]\;.	
\end{multline*}	
Substituting the definition of each $i_{j_k}$ back and using 
$e(X_{t(1)}^{s})=(-1)^{j_{t(1)}}$ this rewrites into  
\begin{multline*}
\textstyle\frac{1}{j_{1}!\cdots j_{p}!}\sum_{s\in S_{n}}e(s)
 \sum_{t\in Sh(2,p-2)}e(t,X_{1}^{s},\ldots,X_{p}^{s})
  (-1)^{(j_{t(p)}-1)}\cdots(-1)^{(j_{t(1)}-1)}\cdot\\
\cdot(-1)^{j_{t(1)}}(j_{t(p)}-1)!\cdots(j_{t(1)}-1)!
 (x_{s(j_{t(p)-1}+j_{t(p)})}\smwedge\cdots\smwedge x_{s(j_{t(p)-1}+1)})
  \smwedge\cdots\smwedge\\
\smwedge(x_{s(j_{t(3)-1}+j_{t(3)})}\smwedge\cdots\smwedge x_{s(j_{t(3)-1}+1)})
   \smwedge\\  
\smwedge[x_{s(j_{t(2)-1}+j_{t(2)})}\smwedge\cdots\smwedge x_{s(j_{t(2)-1}+1)},
  x_{s(j_{t(1)-1}+j_{t(1)})}\smwedge\cdots\smwedge x_{s(j_{t(1)-1}+1)}]
\end{multline*}
and since $e(s)$ as well as $e(t,X_1^s,\ldots,X^s_p)$ keeps properly track of the signs
we can reindex this. After simplification using 
$j_1+\cdots +j_p=n$ we get
\begin{multline*}
\textstyle
 \frac{1}{j_{1}\cdots j_{p}}\sum_{s\in S_{n}}e(s)
  \sum_{t\in Sh(2,p-2)}
(-1)^{j_{t(1)}+n-p}\;
 x_{s(n)}\smwedge\cdots\smwedge x_{s(j_{t(1)}+j_{t(2)}+1)}\smwedge\\
   \smwedge [x_{s(j_{t(1)}+j_{t(2)})}\smwedge\cdots\smwedge x_{s(j_{t(1)}+1)},
  x_{s(j_{t(1)})}\smwedge\cdots\smwedge x_{s(1)}]\;.
\end{multline*}
Since all arguments are actually vectors, we can apply the
symmetric defining expression (\ref{SN_1}) of the Schouten-Nijenuis bracket 
to simplify this further into
\begin{multline*}  
\textstyle(-1)^{n+p}\frac{1}{j_{1}\cdots j_{p}}\sum_{s\in S_{n}}e(s)
 \sum_{t\in Sh(2,p-2)}(-1)^{j_{t(1)}}
  \sum_{q\in Sh(1,j_{t(1)}-1)}\sum_{r\in Sh(1,j_{t(2)}-1)}\\
e(q)e(r)\;x_{s(n)}\smwedge\cdots\smwedge x_{s(j_{t(1)}+j_{t(2)}+1)}
 \smwedge x_{rs(j_{t(1)}+j_{t(2)})}\smwedge\cdots\smwedge x_{rs(j_{t(1)}+2)}
  \smwedge\\
  \smwedge x_{qs(j_{t(1)})}\smwedge\cdots\smwedge x_{qs(2)}\smwedge
   [x_{rs(j_{t(1)}+1)},x_{qs(1)}]\;.
\end{multline*}
Now observe, that for any $s\in S_n$ and shuffle 
$q \in Sh(1,j_{t(1)}-1)$, the permutation 
$(qs(1),\ldots,qs(j_{t(1)}), s(j_{t(1)}+1),\ldots,s(n))$ is again an element of 
$S_n$ and since there are precisely $j_{t(1)}$ many shuffles in 
$Sh(1,j_{t(1)}-1)$ we can just
'absorb' the appropriate sum over shuffles in the previous expression into the sum over
general permutation. The same is true for the shuffles $r \in Sh(1,j_{t(2)}-1)$. 
After reindexing we get:
\begin{multline*} 
\textstyle(-1)^{n+p}\frac{1}{j_{1}\cdots j_{p}}\sum_{s\in S_{n}}e(s)
  \sum_{t\in Sh(2,p-2)}(-1)^{j_{t(1)}}j_{t(1)}j_{t(2)}\cdot\\ 
\cdot x_{s(n)}\smwedge\cdots\smwedge x_{s(j_{t(1)}+2)}\smwedge
 x_{s(j_{t(1)})}\smwedge\cdots\smwedge x_{s(2)}\smwedge[x_{s(j_{t(1)}+1)},x_{s(1)}]=
\end{multline*}
$$
\textstyle(-1)^{n+p-1}\frac{1}{j_{1}\cdots j_{p}}
  \sum_{1\leq l<m\leq p}j_{l}\;j_{m}\;\sum_{s\in S_{n}}e(s)
   x_{s(n)}\smwedge\cdots\smwedge x_{s(3)}\smwedge [x_{s(2)},x_{s(1)}]      
$$
Using this we are able to rewrite the right side of the defining structure equation 
into
\begin{multline*}
\textstyle(-1)^{n+1}\sum_{p=2}^n\frac{(-1)^p}{p!}\sum_{j_1+\ldots+j_p=n}
\frac{1}{j_{1}\cdots j_{p}}
  \sum_{1\leq l<m\leq p}\;j_{l}\;j_{m}\;\sum_{s\in S_{n}}e(s)\cdot\\
\cdot x_{s(n)}\smwedge\cdots\smwedge x_{s(3)}\smwedge [x_{s(2)},x_{s(1)}]
\end{multline*}
and in addition, the left side of the defining structure equation can be rewritten
as 
\begin{multline*}
\textstyle(-1)^{n-2}\frac{(n-2)!}{2(n-2)!}\sum_{s\in S_n}e(s)
 x_{s(n)}\smwedge\cdots\smwedge x_{s(3)}\smwedge [x_{s(1)},x_{s(2)}]=\\
\textstyle(-1)^{n+1}\frac{1}{2}\sum_{s\in S_n}e(s)
 x_{s(n)}\smwedge\cdots\smwedge x_{s(3)}\smwedge [x_{s(2)},x_{s(1)}]\;.
\end{multline*}
Consequently the theorem follows since
\begin{equation}\label{combi_equation}
\textstyle\sum_{p=2}^n\frac{(-1)^p}{p!}\sum_{j_1+\ldots+j_p=n}
 \frac{1}{j_{1}\cdots j_{p}}\sum_{1\leq l<m\leq p}j_{l}\cdot j_{m}=\frac{1}{2}
\end{equation}
for any $n\geq 2$.
(This identity was communicated by Gjergji Zaimi at mathoverflow) To see it
consider the generating function
$$
\textstyle\sum_{p\geq 2} \frac{(-1)^p}{p!}\binom{p}{2}\left(x+\frac{x^2}{2}+\frac{x^3}{3}+\cdots\right)^{p-2}\left(x^2+2x^3+3x^4+\cdots\right)\;.$$
The coefficient of $x^n$ is precisely the left side of (\ref{combi_equation})
and to show that it actually equals $\frac{1}{2}$ use
$$
\begin{array}{ccc}
\textstyle x^2+2x^3+3x^4+\cdots=\frac{x^2}{(1-x)^2}& and &
\frac{e^{-t}}{2}=\sum_{p\geq 2}\frac{(-1)^p}{p!}\binom{p}{2} t^{p-2}
\end{array}
$$
to simplify the generating function to
$$
\textstyle\frac{e^{ln(1-x)}}{2}\frac{x^2}{(1-x)^2}=
 \frac{x^2}{2}+\frac{x^3}{2}+\frac{x^4}{2}+\frac{x^5}{2}+\cdots
$$
\end{proof}
\section{Conclusion and Outlook}
We defined a Lie $\I$-structure on the exterior power $\bwedge\mathfrak{g}_A$
of any Lie Rineard pair, but there is more structure on $\bwedge\mathfrak{g}_A$.
In fact one should look at the interaction of the higher brackets with the
exterior product, to eventually come to some kind of $\I$-Gerstenhaber structure. 
Moreover one should look for generalizations of the Leibniz rule
to the higher brackets.

\begin{appendix}
\section{Lie $\I$-algebras} We recall the most basic stuff
about Lie $\I$-algebras. There are many incarnations of them 
\cite{LV}, \cite{MM}, but we will only look at their 
graded symmetric, 'many brackets' version, since that picture fits nicely into
the Schouten calculus and is moreover useful when it comes to actual computations. 

Lie $\I$-algebras are defined on $\Z$-graded vector spaces and 
consequently we recall them first:
\subsection{Graded Vector Spaces}
In what follows $\mathbb{K}$ will always be a field and $\Z$ the Abelian
group of integers with respect to addition. 
A $\Z$-graded $\mathbb{K}$-vector space $V$ is the direct sum 
$\oplus_{n\in \Z} V_n$ of $\mathbb{K}$-vector spaces $V_n$.  
Since this is a coprodut, there are natural injections 
$i_n:V_n \to V$ and a vector is called \textbf{homogeneous of degree} $n$ 
if it is in the image of the injection $i_n$. 
In that case we write $deg(v)$ or $|v|$ for its degree.

According to a better readable text we just write graded vector space 
as a shortcut for $\Z$-graded $\mathbb{K}$-vector space.

A morphism $f : V \to W$ of graded vector spaces, homogeneous of degree $r$,  
is a sequence of linear maps $f_n : V_n \to W_{n+r}$ for any $n\in \Z$ and
the integer $r\in\Z$ is called the degree of $f$, denoted by $deg(f)$ (or $|f|$). 

For any $n\in\N$, an $n$-multilinear map $f: V_1 \times \cdots \times V_n \to W$,
homogeneous of degree $r$ is a sequence of $n$-multilinear maps 
$f_{k} : (V_1)_{n_1} \times \ldots \times (V_k)_{n_k} \to W_{\sum n_i+r}$ for all $j_i\in \Z$ with $\sum j_i=k$.

The $\Z$-graded tensor product $V \otimes W$ of two graded vector spaces 
$V$ and $W$ is given by
$$
\textstyle\left(V \otimes W \right)_n :=
\oplus_{i+j=n}\left( V_i \otimes W_j\right)
$$
and the Koszul commutativity constraint
$\tau: V \otimes W \to  W \otimes V$ is on homogeneous elements 
$v\otimes w \in V \otimes W$ defined 
by 
$$\tau(v \otimes w):=(-1)^{deg(v)deg(w)} w \otimes v$$ and then extended to
$V\otimes W$ by linearity.

\begin{remark}
We define the symbols $e(v):=(-1)^{deg(v)}$, 
$e(v,w):=(-1)^{deg(v)deg(w)}$. The \textbf{Koszul sign} 
$e(s;v_1,\ldots,v_k) \in \{-1,+1\}$ is defined for any permutation $s\in S_k$ and
any homogeneous vectors $v_1,\ldots,v_k\in V$ by 
\begin{equation}\label{Koszul_convention}
v_1\otimes \ldots \otimes v_k= e(s;v_1,\ldots,v_k) 
 v_{s(1)}\otimes \ldots \otimes v_{s(k)}.
\end{equation}
In an actual computation it can be determined by the following rules: 
When a permutation $s\in S_k$ is a transposition  $j\leftrightarrow j+1$
of consecutive neighbors, 
then $e(s;v_1,\ldots,v_k)= (-1)^{deg(v_j)\cdot deg(v_{+1})}$ and 
if $t\in S_k$ is another permutation, then
$e(ts;v_1,\ldots,v_k)=e(t;v_{s(1)},\ldots,v_{s(k)})e(s;v_1,\ldots,v_k)$.
\end{remark}
A graded $k$-linear morphism $ f: \bigtimes^k V \to W$ is called 
\textbf{graded symmetric} if 
$$f(v_1,\ldots,v_k) = e(s;v_1,\ldots,v_k)f(v_{s(1)},\ldots,v_{s(k)})$$
for all $s\in S_k$. 
\subsection{Shuffle Permutation}
Let $S_k$ be the symmetric group, i.e the group of all bijective maps 
of the ordinal $\ordinal{k}$.
\begin{definition}[Shuffle Permutation] For any $p,q\in \N$
a $(p,q)$-shuffle is a permutation 
$s\in S_{p+q}$ with $s(1)<\ldots<s(p)$ and $s(p+1)<\ldots<s(p+q)$.
We write $Sh(p,q)$ for the set of all $(p,q)$-shuffles.

More generally for any $p_1,\ldots,p_n\in\N$ a $(p_1,\ldots,p_n)$-shuffle
is a permutation $s\in S_{p_1+\cdots+p_n}$ with  
$s(p_{j-1}+1)<\ldots<s(p_{j-1}+p_{j})$. 
We write $Sh(p_1,\ldots,p_n)$ for the set of all $(p_1,\ldots,p_n)$-shuffles.
\end{definition}
For more on shuffles, see for example at \cite{RS}.
\subsection{Lie $\I$-algebas} On the structure level Lie $\I$-algebras generalize 
(differential graded) Lie-algebras to a setting where the Jacobi identity isn't 
satisfied any more, but holds up to particular higher brackets. This can be defined 
in many different ways \cite{LV}, 
but the one that works best for us is its 'graded symmetric, many bracket' version.
\begin{definition} A \textbf{Lie $\I$-algebra} $\left(V,(D_k)_{k\in \N}\right)$
is a $\Z$-graded $\R$-vector space $V$, together with a sequence $(D_k)_{k\in \N}$ of
graded symmetric, $k$-multilinear maps $D_k : \bigtimes^k V \to V$, 
homogeneous of of degree $-1$, such that the \textbf{weak Jacobi equations}
$$\label{sh_Jacobi}
\textstyle\sum_{i+j=n+1} \left(\sum_{s\in Sh(j,n-j)}
e\left(s;v_1,\ldots,v_n\right)D_i\left(D_j \left(
v_{s_1}, \ldots, v_{s_j} \right), v_{s_{j+1}}, \ldots, v_{s_n}\right)\right) = 0
$$
are satisfied for any integer $n \in \N$ and any vectors $v_1,\ldots,v_n \in V$. 
\end{definition} 
In particular Lie $\I$-algebras generalizes ordinary Lie algebras, if the grading is chosen 
right:
\begin{example}[Lie Algebra] Every Lie algebra $\left(V,[\cdot,\cdot]\right)$ is a  
Lie $\I$-algebra if we consider $V$ as concentrated in degree one and define $D_k=0$ for any $k \neq 2$ as well as $D_2(\cdot,\cdot):=[\cdot,\cdot]$.
\end{example}
Very different from common Lie theory is, that a morphism
of Lie $\I$-algebras is not necessarily just a single map. In fact such a 
morphism is a \textit{sequence} of maps, satisfying a particular structure
equation. To understand how these morphisms emerge, look for example at \cite{MM}.
\begin{definition}For any two Lie $\I$-algebras $(V,(D_k)_{k\in\N})$ and 
$(W,(l_k)_{k\in\N})$ a \textbf{morphism of Lie $\I$-algebras} is a sequence
$(f_k)_{k\in\N}$ of graded symmetric, $k$-multilinear maps
$$
f_k: V \times \cdots \times V \to W
$$
homogeneous of degree $0$, such that the structure equation
\begin{multline}\label{lie-infty-morph}
\textstyle\sum_{p+q=n+1}\left({\sum_{s\in Sh(q,p-1)}e\left(s\right)}
f_p\left(D_{q}\left(v_{s(1)},\ldots,v_{s(q)}\right),
v_{s(q+1)},\ldots, v_{s(n)}\right)\right)=\\
\textstyle\sum_p\frac{1}{p!}\sum^{k_1+\ldots +k_p =n}_{s\in Sh(k_1,\ldots,k_p)}
e\left(s\right)
l_p\left(f_{k_1}\left(v_{s(1)},\ldots,v_{s(k_1)}\right),
\ldots,
f_{k_p}\left((v_{s(n-k_p+1)},\ldots,v_{s(n)}\right)\right)
\end{multline}
is satisfied for any $n\in\N$ and any vectors $v_1,\ldots,v_n\in V$.

The morphism is called \textbf{strict}, 
if in addition $f_k=0$ for all $k \geq 2$,that is, if the morphism is a
single map, that commutes with all brackets. 
\end{definition}
\end{appendix}

\end{document}